\documentclass[11pt,a4paper]{article}
\usepackage[T1]{fontenc}
\usepackage[utf8]{inputenc}
\usepackage{authblk}

\usepackage[margin=1in]{geometry}  
\usepackage{graphicx}              
\usepackage{amsmath}               
\usepackage{amsfonts}              
\usepackage{amsthm}                

 \newtheorem{theorem}{Theorem}[section]

\theoremstyle{remark}

\title{On the stability of Forward in Time and Centred in Space
  (FTCS) scheme for scalar hyperbolic equation\footnote{Carried out work is
    financially supported by DST India through project \#
    SR/FTP/MS-015/2011}}
\author{Ritesh Kumar Dubey\footnote{\texttt{mail-to: riteshkd@gmail.com}}, \\
Research Institute, SRM University, Tamilnadu\\ India}
\date{}
\begin{document}
\maketitle

\begin{abstract}
It is well known that Forward Time and Centred in Space (FTCS) scheme
for scalar Hyperbolic Conservation Law (HCL) is unconditionally
unstable. The main contribution of this work to show that FTCS is
conditionally stable for HCL. A new approach is used to give bounds on the
initial data profile by transforming FTCS into two point convex
combination scheme. Numerical results are given in support of
the claim.
\end{abstract}

{\bf keyword}
Numerical oscillations; Von-Neumann stability; smoothness parameter; finite difference schemes; hyperbolic equations.

\section{Introduction}\label{sec1}
We consider for the simplicity, the linear transport problem,
\begin{eqnarray}
\frac{\partial }{\partial t}u(x,t) + a \frac{\partial}{\partial x}u(x,t)&=&0, a\in \mathbb{R} \label{eq1sec1} \\
u_{0}(x)&=&f(x) 
\end{eqnarray}
where $u=u(x,t)$ is a scalar field transported by flow of constant
velocity $a$. The exact solution of linear transport equation \label{eqsec1} is given by
\begin{equation}
u(x,t)=u_{0}(x\pm at). \label{eq2sec1}
\end{equation}
Seemingly simple, the linear transport equation (\ref{eq1sec1}) has
played a crucial role in the development of the numerical methods for
general hyperbolic conservation laws.  It is the closed form solution (\ref{eq2sec1})
of (\ref{eq1sec1}) which pave the  way to analyze any new numerical
scheme devised for more complex non-linear hyperbolic problems for its
convergence and stability.  Considering the solution (\ref{eq2sec1}),
it is natural to seek a numerical scheme which is stable in the sense
that it does not induce numerical oscillations. Therefore, the notion
of stability of a numerical schemes evolves around the induced
spurious numerical oscillations.  In order to elaborate it, we
consider a uniform grid with the spatial width $h$, time step $k$ and
denote the discrete mesh point $(x_{j},t_{n})$ by $x_{j}=jh,\; j\in
\mathbb{Z}$ and $t_{n}=nk,\; k\in \mathbb{N}$.

In the seminal work Courant-Friedrichs and Levy shown that for the
convergence a difference scheme must contains the physical domain of
dependence of partial differential equation \cite{CFL1928}. In other
words, they gave a necessary condition on ratio of the spatial and
time discretization step for the convergence of the difference scheme
known as CFL number. Since then, the notion of CFL condition has been
an indispensable tool for defining the stability of numerical schemes.
It is known that the linear Von-Neumann stability analysis of a
numerical scheme for (\ref{eq1sec1}) deduces a stability condition on
the CFL number $C=|a|\lambda,\; \lambda=\frac{k}{h}$.  Moreover,
almost all notion of non-linear stability such as upwind range condition, monotone stability
\cite{Lax1954,Majda1980,Sanders1983}, positivity preserving
\cite{Jameson95positiveschemes,LaxLiu1998}, total variation stability
\cite{harten1983} heavily relies on CFL number to devise any new
stable scheme e.g., TVD schemes in \cite{Rkd1,Jerez2014292,rider},
essentially non-oscillatory (ENO) schemes \cite{harten1987, shu1988}, weighted
essentially non-oscillatory schemes \cite{Zahran200937}.  Apart
from the notion of stability, in a recent paper CFL number is
exploited for improved approximation by the flux limiters based scheme
in \cite{Kemm}. In this work, we are interested in the
stability of the following FTCS scheme. It is obtained by the
discretization of (\ref{eq1sec1}) by replacing the time derivative
with a forward difference, and the space derivative with a centred
difference formula i.e.,
\begin{equation}
u_{j}^{n+1}=u_{j}^{n}-\frac{a\,k}{2h}\left(u_{j+1}^{n}-u_{j-1}^{n}\right). \label{ftcs1}
\end{equation}
Above three point centred FTCS scheme (\ref{ftcs1}) seems to be a correct and natural choice as
the spatial discretization in FTCS does not violets the physical
domain of dependence of (\ref{eq1sec1}) given in \cite{CFL1928}. It is
also interesting to note that the FTCS (\ref{ftcs1}) and the centered
Lax-Wendroff(LxW) scheme \cite{LxW1960} shares the same spatial
stencil of grid points. Note that for the CFL number
$C\leq 1$ the three point centred LxW scheme is linearly stable
\cite{Toro}. Contrary to the expectation, the solution
obtained by FTCS scheme (\ref{ftcs1}) is diverging and and induced
oscillations grow exponentially no matter how small the time step is
compared to the space step. The classical Von-Neumann stability
analysis also shows that FTCS (\ref{ftcs1}) is unconditionally
unstable. Moreover, FTCS does not satisfies any of the notion of non-linear stability mentioned above see \cite{Laney}.
\section{Non-oscillatory condition on FTCS}
The above, unconditional unstability FTCS scheme (\ref{ftcs1}) is
surprising as it can be observed that for smooth initial data, such as
sinusoidal wave, the induced oscillations by FTCS does not grow
immediately. Moreover, the magnitude and occurrence of induced
oscillations can be controlled by choosing small CFL number
$C=\lambda |a|$, see Figure \ref{Fig1}. On the other hand when
applied on discontinuous data, FTCS introduces strong oscillations
immediately see Figure \ref{Fig2}(a).  These observations have been the motivation for the
present study on the dependence of induced oscillations by FTCS scheme
on data type and CFL number. In the following result, a classification
on data type is given in terms of ratio of consecutive gradients which
is defined as
\begin{equation}
\displaystyle
\theta_{j}^{n}=\left\{\begin{array}{ccc}\frac{\Delta_{-}u_{j}^{n}}{\Delta_{+}u_{j}^{n}}
& \mbox{if} &a\geq0,\\ \frac{\Delta_{+}u_{j}^{n}}{\Delta_{-}u_{j}^{n}} &
\mbox{if} &a<0,\end{array}\right.
\label{smtprm}
\end{equation}
$\Delta_{\pm}u_{j}^{n}=\pm u_{j \pm\, 1}\mp\, u_{j}$. We assume that
the time and space step ratio satisfies the CFL condition $a\lambda \leq
1$ and denote the sign of $x$ by $sgn(x)$.
\begin{theorem}\label{thm1}
In the solution data region where $\displaystyle \theta_{j}^{n} \in
\mathcal{S}_{FTCS}= \left(\infty, -1\right]\cup\left[\frac{sgn(a)|a|\lambda}{2-sgn(a)|a|\lambda},\infty\right)$, the FTCS
scheme (\ref{ftcs1}) is stable and does not introduce numerical oscillations
provided CFL number $a\lambda\leq 1$.
\end{theorem}
\begin{proof}
Rewrite FTCS (\ref{ftcs1}) in the form
\begin{equation}
u_{j}^{n+1}=u_{j}^{n}-\frac{a\lambda}{2}\left(\Delta_{+}u_{j}^{n}+\Delta_{-}u_{j}^{n}\right). \label{ftcs2}
\end{equation}
\begin{itemize}
\item {\bf Let $a>0$:}  FTCS (\ref{ftcs2}) can be written as,
\begin{equation}
u_{j}^{n+1}=u_{j}^{n}-\frac{a\lambda}{2}\left(\frac{\Delta_{+}u_{j}^{n}}{\Delta_{-}u_{j}^{n}}+1\right) \Delta_{-}u_{j}^{n}. \label{ftcs3}
\end{equation}
which can be rewritten in upwind stencil form as,
\begin{equation}
u_{j}^{n+1}= \alpha^{+}\, u_{j}^{n} + \beta^{+}\, u_{j-1}^{n}. \label{ftcs4}
\end{equation}
where the coefficients $
\displaystyle \alpha^{+}=1-\frac{a\lambda}{2}\left(\frac{\Delta_{+}u_{j}^{n}}{\Delta_{-}u_{j}^{n}}+1\right)$
and $\displaystyle \beta^{+}
=\frac{a\lambda}{2}\left(\frac{\Delta_{+}u_{j}^{n}}{\Delta_{-}u_{j}^{n}}+1\right)$.
Note that $\alpha +\beta =1$, thus to ensure non-oscillatory stable
approximation such that $u_{j-1}^{n}\leq u_{j}^{n+1} \leq u_{j}^{n}$
by (\ref{ftcs4}), it is sufficient that
\begin{equation}
\alpha^{+}=
1-\frac{a\lambda}{2}\left(\frac{\Delta_{+}u_{j}^{n}}{\Delta_{-}u_{j}^{n}}+1\right)
\geq 0,\;\mbox{and}\; \beta^{+}
=\frac{a\lambda}{2}\left(\frac{\Delta_{+}u_{j}^{n}}{\Delta_{-}u_{j}^{n}}+1\right)\geq
0. \label{inq1}
\end{equation}
Inequalities (\ref{inq1}) satisfies if,
$$ \frac{\Delta_{+}u_{j}^{n}}{\Delta_{-}u_{j}^{n}} \leq \frac{2-a\lambda}{a\lambda} \; \mbox{and}\; \frac{\Delta_{+}u_{j}^{n}}{\Delta_{-}u_{j}^{n}}\geq -1.$$
Which on inversion yield non-oscillatory condition for FTCS scheme (\ref{ftcs1}) in case of $a> 0$,
\begin{equation}
\frac{\Delta_{-}u_{j}^{n}}{\Delta_{+}u_{j}^{n}} \leq -1 \; \mbox{OR}\; \frac{\Delta_{-}u_{j}^{n}}{\Delta_{+}u_{j}^{n}}\geq \frac{a\lambda}{2-a\lambda}\label{inq1}
\end{equation} 
\item{\bf Let $a<0$}, then (\ref{ftcs2}) can be written in upwind
  stencil form as,
\begin{equation}
u_{j}^{n+1}= \alpha^{-}\, u_{j}^{n} + \beta^{-}\, u_{j+1}^{n}. \label{ftcs5}
\end{equation}
where $\displaystyle \alpha^{-}= \left(1+\frac{a\lambda}{2}\left(1+ \frac{\Delta_{-}u_{j}^{n}}{\Delta_{+}u_{j}^{n}}\right)\right)$ and
$\displaystyle \beta^{-}= \frac{-a\lambda}{2}\left(1+ \frac{\Delta_{-}u_{j}^{n}}{\Delta_{+}u_{j}^{n}}\right)$.  Similar to above,
(\ref{ftcs5}) ensures for a non-oscillatory approximation provided 
\begin{equation}\left(\frac{a\lambda}{2}\left(1+
\frac{\Delta_{-}u_{j}^{n}}{\Delta_{+}u_{j}^{n}}\right)\right)\geq -1\;
\mbox{and} \frac{-a\lambda}{2}\left(1+
\frac{\Delta_{-}u_{j}^{n}}{\Delta_{+}u_{j}^{n}}\right)\geq 0
.\label{inq4}\end{equation}

Note that $a\lambda<0$, therefore inequalities in (\ref{inq4}) satisfy if
$$\left(1+\frac{\Delta_{-}u_{j}^{n}}{\Delta_{+}u_{j}^{n}}\right)\leq \frac{-2}{a\lambda}\;
\mbox{and}\;1+\frac{\Delta_{-}u_{j}^{n}}{\Delta_{+}u_{j}^{n}}\geq 0$$
or
$$\left(\frac{\Delta_{-}u_{j}^{n}}{\Delta_{+}u_{j}^{n}}\right)\leq \frac{-(2+a\lambda)}{a\lambda}\;
\mbox{and}\;\frac{\Delta_{-}u_{j}^{n}}{\Delta_{+}u_{j}^{n}}\geq -1$$
Since $\frac{-(2+a\lambda)}{a \lambda}> 0$, thus on inversion compound inequality satisfies if,
\begin{equation}
\frac{\Delta_{+}u_{j}^{n}}{\Delta_{-}u_{j}^{n}}\leq -1 \; \mbox{OR}\;
\frac{\Delta_{+}u_{j}^{n}}{\Delta_{-}u_{j}^{n}} \geq
\frac{-a\lambda}{(2+a\lambda)}\label{inq3}
\end{equation} 
\end{itemize}
\end{proof}
The non-oscillatory stable region $\mathcal{S}_{FTCS}$ given by
inequalities (\ref{inq1}) is given for $a>0$ in Figure (\ref{region})
which clearly shows the effect of CFL number on the stability of FTCS.
\subsection{Non-oscillatory bound for non-linear scalar problem}
Consider the non-linear scalar hyperbolic conservation laws
\begin{equation}
u(x,t)_{t} + g(u(x,t))_{x}=0,\; u(x,0)=u_{0}(x).  \label{nonlineq}
\end{equation}
The characteristic speed $g'(u)$ associated with (\ref{nonlineq}) at
interface $x_{j+\frac{1}{2}}$ of the cell $[x_{j},x_{j+1}]$ can be approximated by 
\begin{equation}
\alpha_{j+\frac{1}{2}} =\left\{\begin{array}{cc}
\frac{\Delta_{+}g_{j}^{n}}{\Delta_{+}u_{j}^{n}} & \mbox{if}\;
\Delta_{+}u_{j}^{n}\neq 0,\\ g'_{j} & else. \end{array}\right. \label{spd}
\end{equation}
For (\ref{nonlineq}), FTCS can be written as 
\begin{equation}
u_{j}^{n+1}=u_{j}^{n}-\frac{k}{2h}\left(g_{j+1}^{n}-g_{j-1}^{n}\right). \label{nlftcs1}
\end{equation}
where $g_{j}^{n}=g(u_{j}^{n})$. Which using (\ref{spd}) can be written as  
\begin{equation}
u_{j}^{n+1}=u_{j}^{n}-\frac{k}{2h}\left(a_{j+\frac{1}{2}}u_{j+1}^{n}-a_{j-\frac{1}{2}}u_{j-1}^{n}\right). \label{nlftcs2}
\end{equation}
In the non-sonic region, following result give bounds for
non-oscillatory approximation of non-linear scalar problem
(\ref{nonlineq}) by FTCS.
\begin{theorem}\label{thm2}
Away from the sonic point i.e., where $a_{j-\frac{1}{2}}\times
a_{j+\frac{1}{2}}>0 $, the FTCS scheme (\ref{nlftcs2}) is stable and
does not introduce numerical oscillations in the solution region 
\begin{equation}
\displaystyle \theta_{i}^{n} \in \mathcal{S}_{FTCS} =
\left\{\begin{array}{lc}\left(\infty,
-\frac{\alpha_{j+{1}{2}}}{\alpha_{j-{1}{2}}}\right]\cup\left[\frac{\lambda
      \alpha_{j+{1}{2}}}{2-\lambda \alpha_{j-{1}{2}}}, \infty\right) &
    \mbox{if}\; g'(u)>0,\\ & \\
\left(\infty,
-\frac{\alpha_{j-{1}{2}}}{\alpha_{j+{1}{2}}}\right]\cup\left[-\frac{\lambda
      \alpha_{j-{1}{2}}}{2+\lambda \alpha_{j+{1}{2}}}, \infty\right) &
    \mbox{if}\; g'(u)<0.
\end{array}
\right.
\end{equation}
provided $\lambda \max_{u}|g'(u)|\leq 1$.
\end{theorem}
\begin{proof}
Analogous to the proof of Theorem (\ref{thm1}). 
\end{proof}
\section{Numerical Results}
\subsection{Linear Case}
Consider linear transport equation (\ref{eqsec1}) along with the
following initial conditions
\begin{equation} u_{0}(x)=\sin(\pi\,x),\; x\in [-1:1] \label{IC1} \end{equation} 
\begin{equation} 
u_{0}(x)=\left\{\begin{array}{cc} 1 &\;\mbox{if}\; |x|\leq 1/3\\ 0 &\;\mbox{else}\;\end{array}\right., x\in [-1:1] \label{IC2}
\end{equation}
\begin{equation} 
u_{0}(x)=\left\{\begin{array}{cc} \exp\left(\frac{-1}{1-x^2}\right) &\;\mbox{if}\; |x|\leq 1\\ 0 & \;\mbox{else}\;\end{array}\right. x\in [-2:4]  \label{IC3}
\end{equation}
In Figure \ref{Fig1}, numerical results corresponding to initial
condition (\ref{IC1}) are given to show the effect of CFL number on
induced oscillations by FTCS. It can be clearly seen that numerical
oscillations disappear for small CFL number. This supports the result
in Theorem \ref{thm1} as $C\rightarrow 0$ reduce the region of
oscillations of $\theta$ to (-1,0) as shown in Figure \ref{region}.

In order to show data region which cause induced oscillations by FTCS,
the following hybrid approach using first order
non-oscillatory upwind scheme is used {\bf to show},
\begin{itemize}
\item {\it FTCS does not introduce oscillations for data region $
    \mathcal{S}_{FTCS}$:} Use upwind scheme if $\theta^{n}_{j} \notin (-1,
  \frac{2-C}{C})$ other wise FTCS is used. Results obtained by this approach (using legend 'FTCSUP')
  in Figure \ref{Fig2}(b) and Figure \ref{Fig3}(a) show that
  oscillations does not arise is shown as FTCSUP.
\item {\it FTCS introduces oscillations only for data region
    $\mathbb{R}\setminus\mathcal{S}_{FTCS}$:} If $\theta^{n}_{j} \in
  (-1, \frac{2-C}{C})$ use FTCS, other wise first order upwind
  scheme is used. Results obtained by this approach in Figure
  \ref{Fig2}(c) and Figure \ref{Fig3}(b) (using legend 'FTUPCS')
  show occurrence of spurious oscillations by FTUPCS.
\end{itemize}
\subsection{Nonlinear scalar}
Consider the non-linear Burgers equations
\begin{equation}
u_{t} + (\frac{u^2}{2})_{x}=0,\label{burgerseq}
\end{equation}
with initial conditions 
\begin{equation}
u_{0}(x)=\left\{\begin{array}{cc}1 & \mbox{if}\;x\leq 0.5\\0 &
x>0.5,\; x\in [0,1].\setminus{1}\end{array}\right.\label{test1IC}
\end{equation}
\begin{equation}
u_{0}(x)=\left\{\begin{array}{cc}1 & \mbox{if}\;x=1\\0, & x\in [0,2]\setminus{1}.\end{array}\right.\label{test2IC}
\end{equation}
In Figure \ref{burgerFig1} and \ref{burgerFig2} numerical results are
given corresponding to initial condition (\ref{test1IC}) and
(\ref{test2IC}) respectively. Similar to the linear case, again
results by FTCS and FTUPCS are oscillatory whereas results by FTCSUP
does not show any numerical oscillations. These results confirm that
FTCS is conditionally stable and introduce the oscillations only in the
solution data region where $\theta \in \mathcal{S}_{FTCS}\setminus
\mathbb{R}$

\section{Conclusion} Non-oscillatory stability bounds on initial data profile are given and
numerically tested for unconditionally unstable (in Von-Neumann sense)
FTCS scheme. An extension of this preliminary study for other existing
schemes for non-linear hyperbolic problem is being carried out for a
separate work.

\newpage
\begin{figure}
\begin{center}
\includegraphics[scale=0.5]{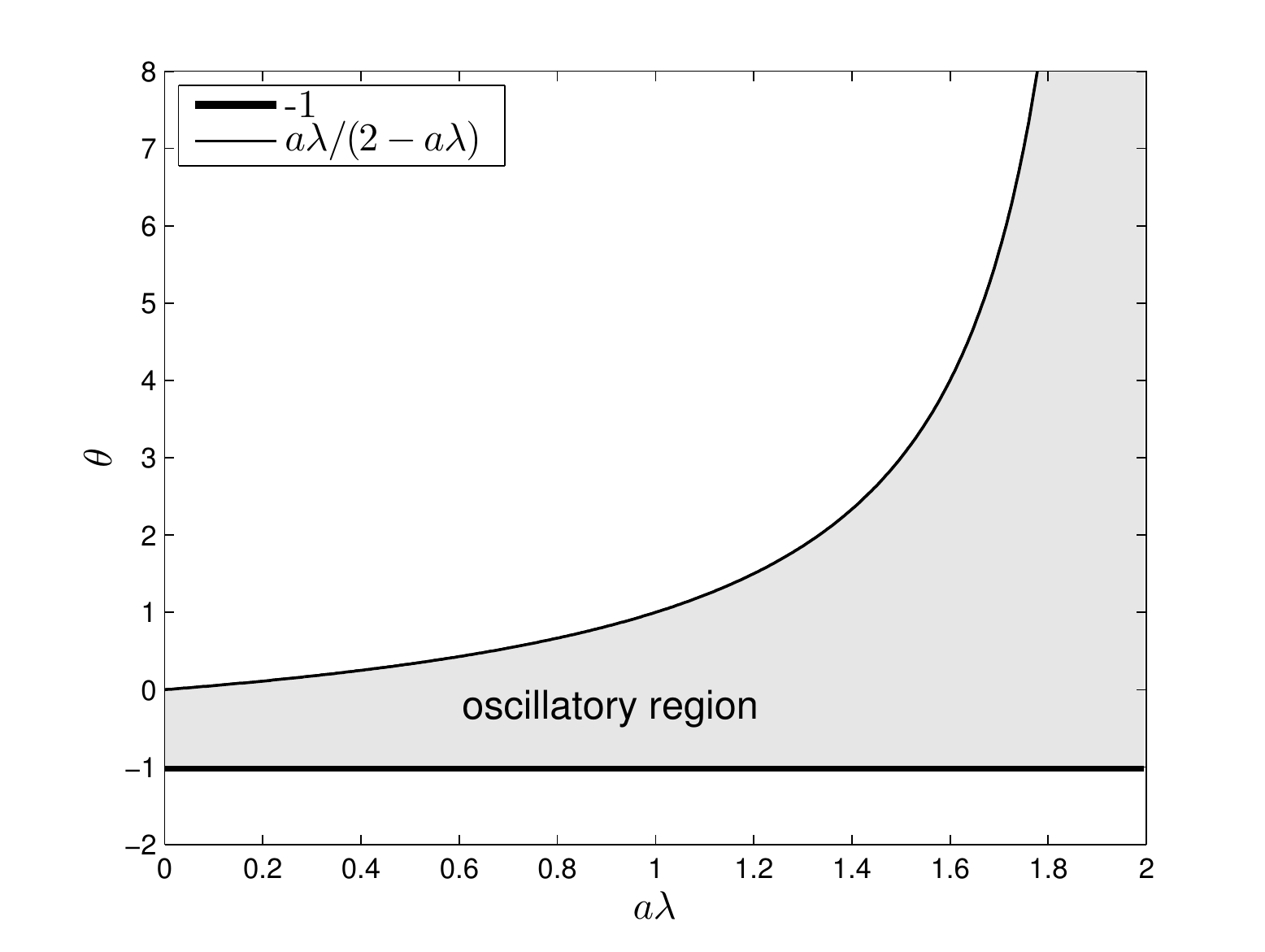}
\end{center}
\caption{\label{region} Plot $\theta$ v/s $a\lambda, \;a=1$: $a\lambda
  \rightarrow 0$ stabilizes the FTCS as it corresponds to reduced
  oscillatory shaded region.}
\end{figure}

\begin{figure}
\begin{center}
\begin{tabular}{ccc}
  \includegraphics[scale=0.35]{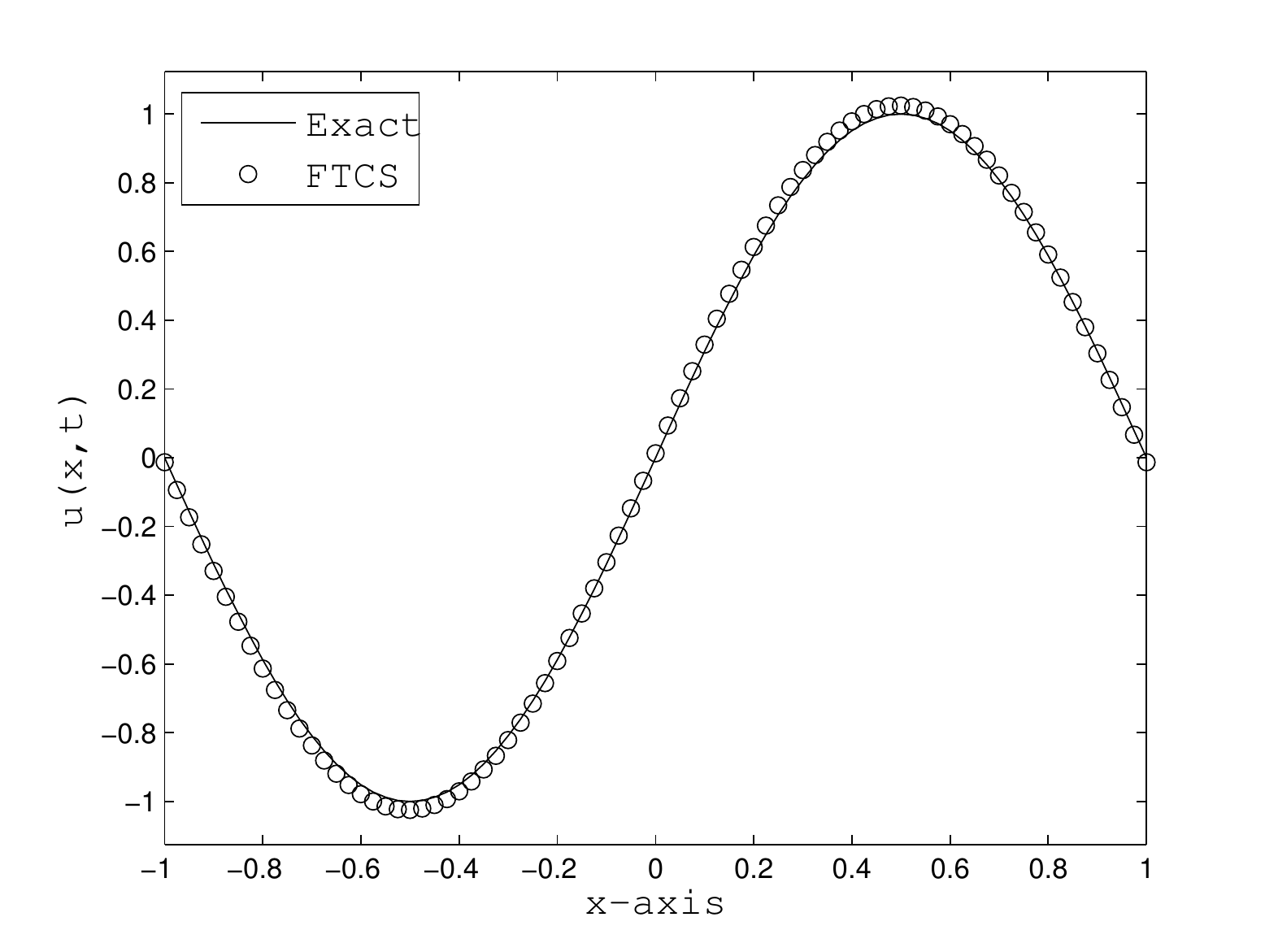}& \includegraphics[scale=0.35]{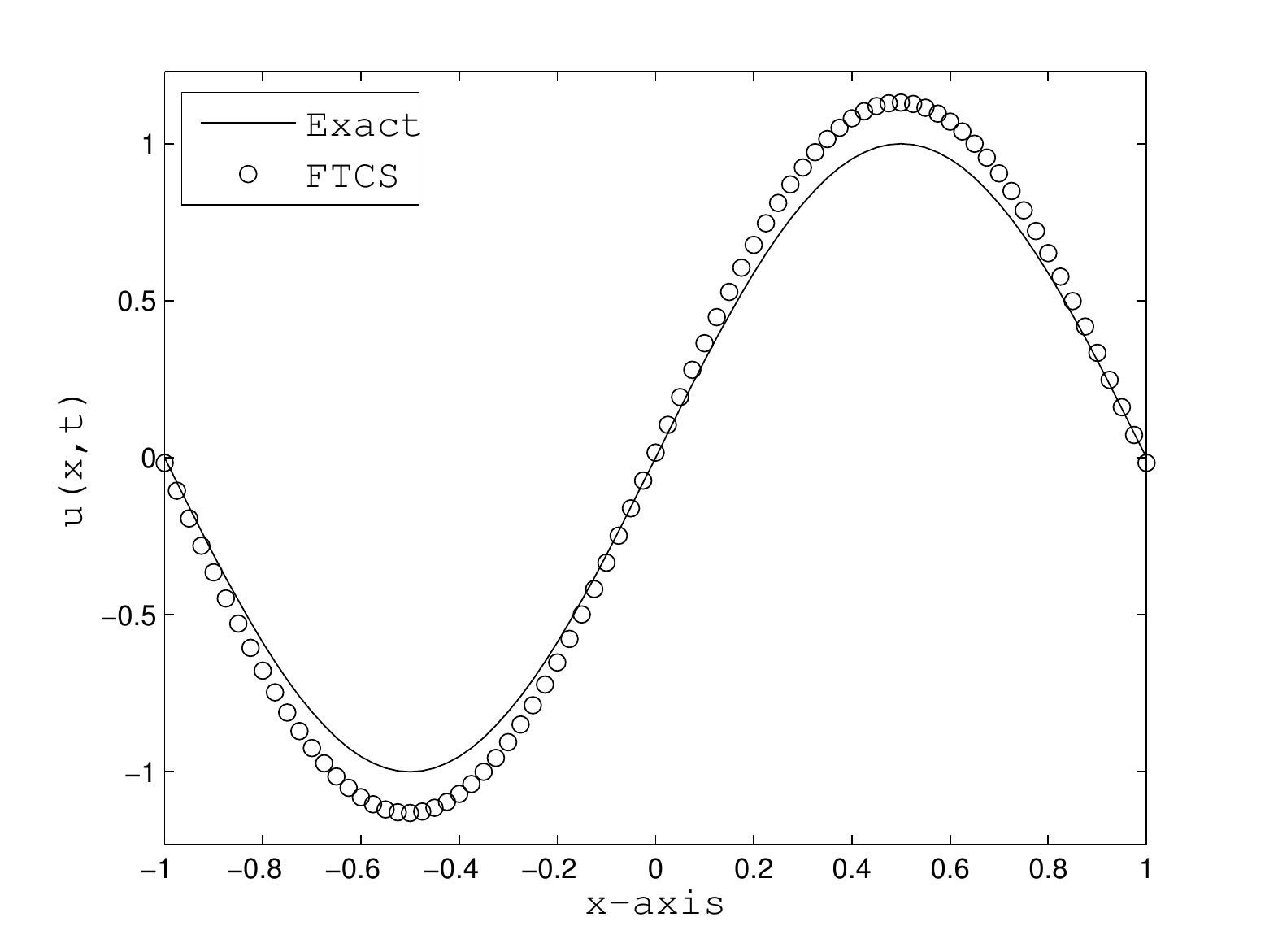}& \includegraphics[scale=0.35]{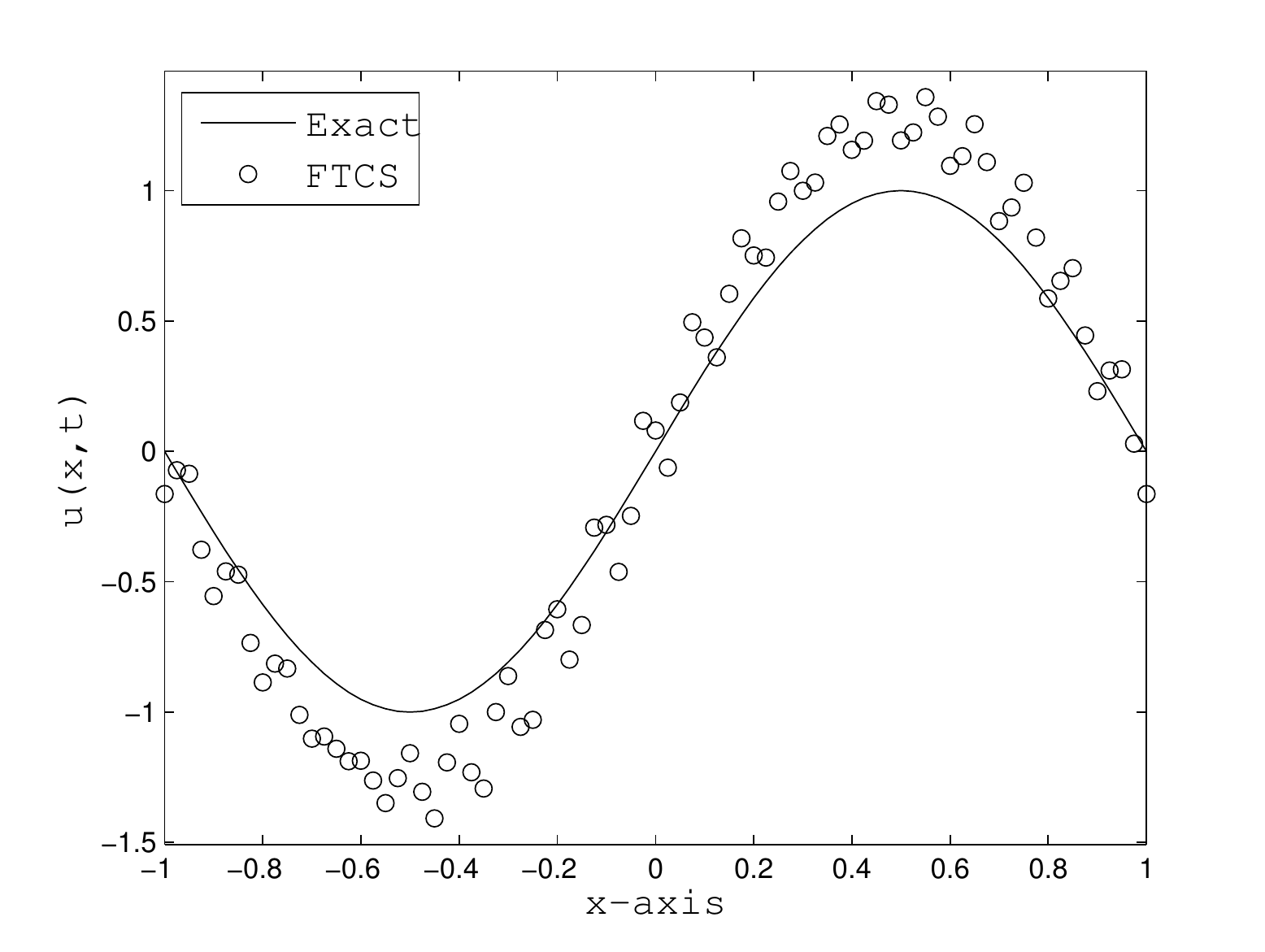}\\
  (a) & (b) & (c)
\end{tabular}
\caption{\label{Fig1} Effect of CFL on induced oscillations in the
  solution corresponding to (\ref{IC1}), at $T=4$ using $N=80$ grid
  points (a) $C=0.05$ (b) $C=0.25$ and (c) $C=0.5$.}
\end{center}
\end{figure}
\begin{figure}
\begin{center}
\begin{tabular}{ccc}
\includegraphics[scale=0.4]{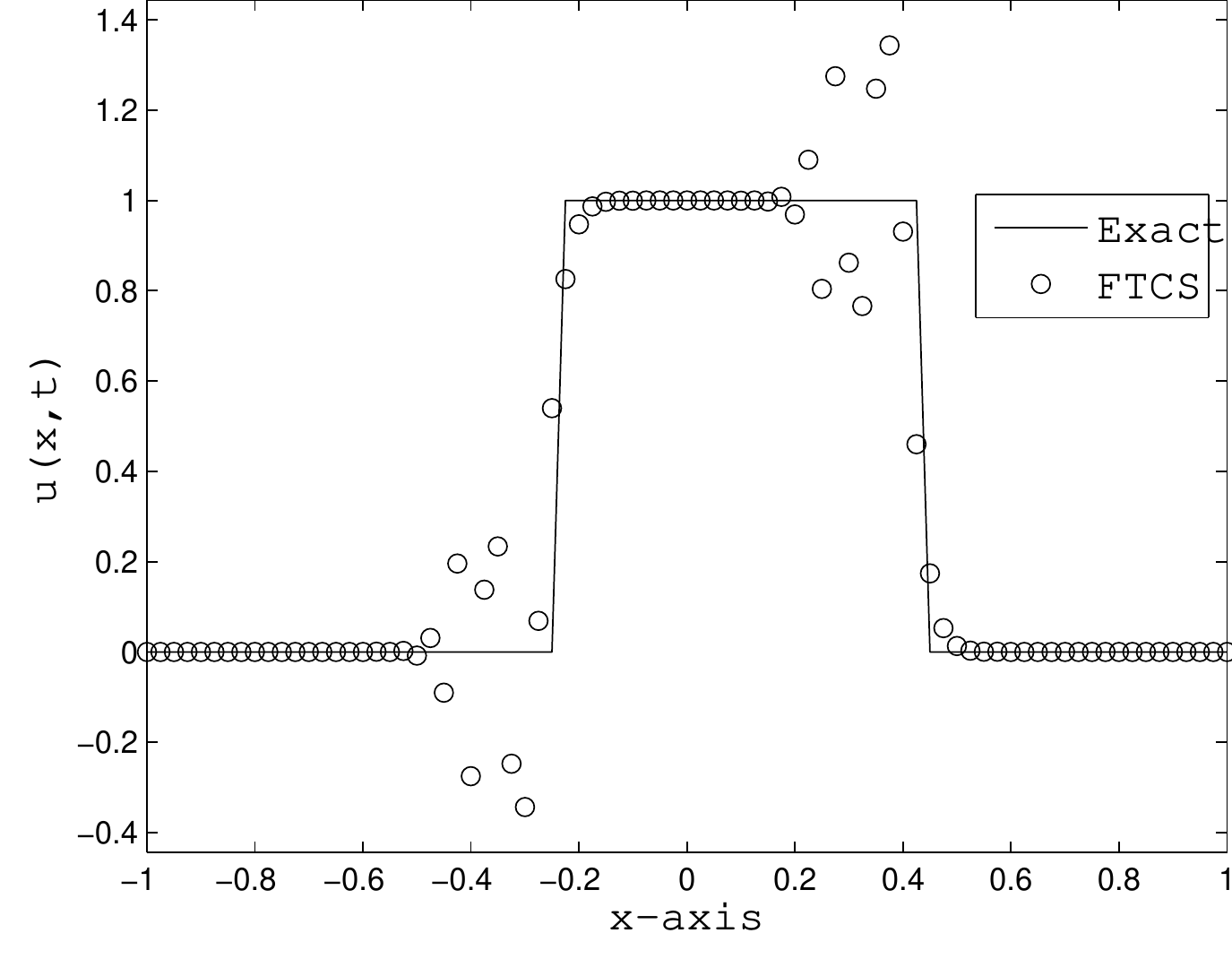}& \includegraphics[scale=0.4]{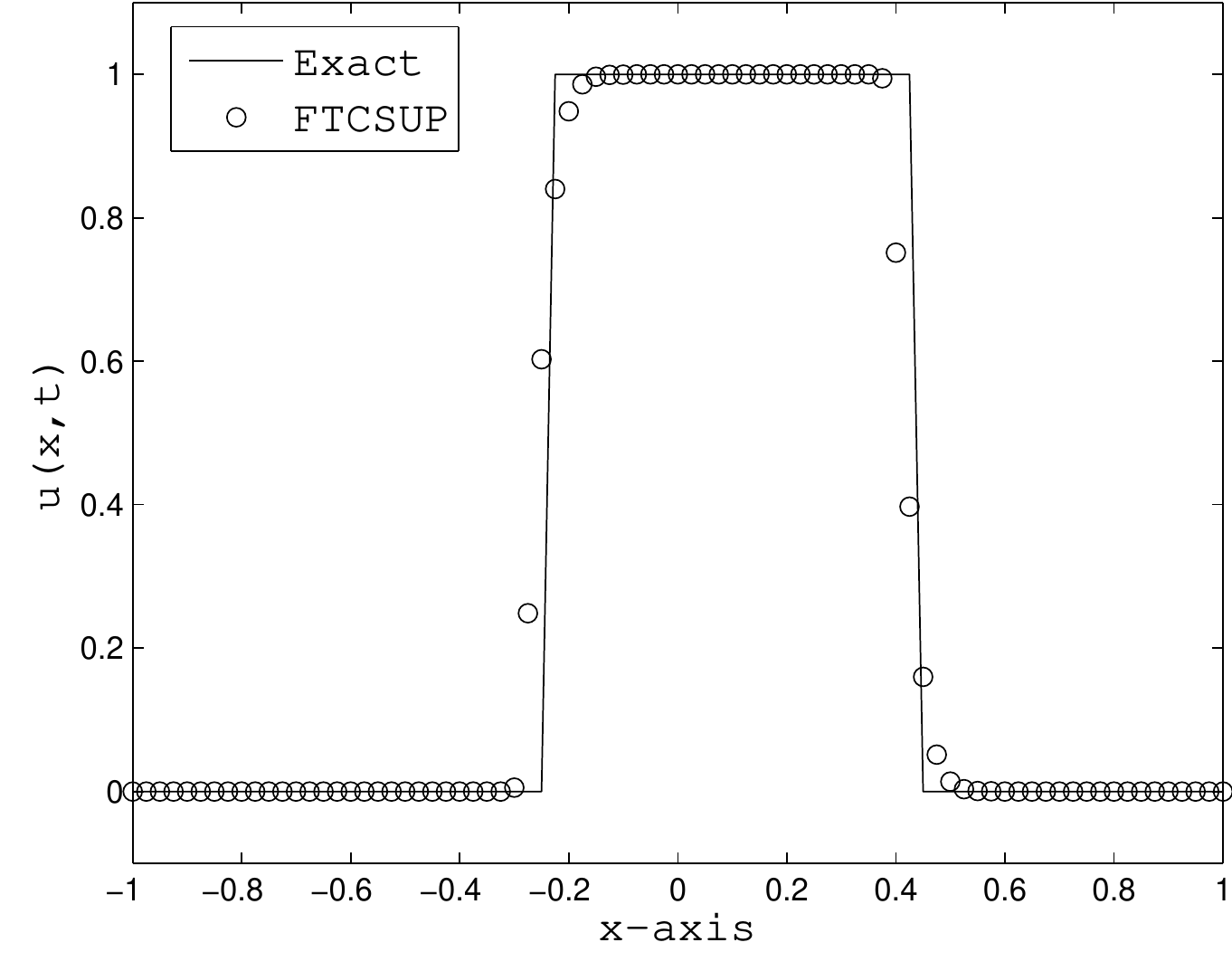}& \includegraphics[scale=0.4]{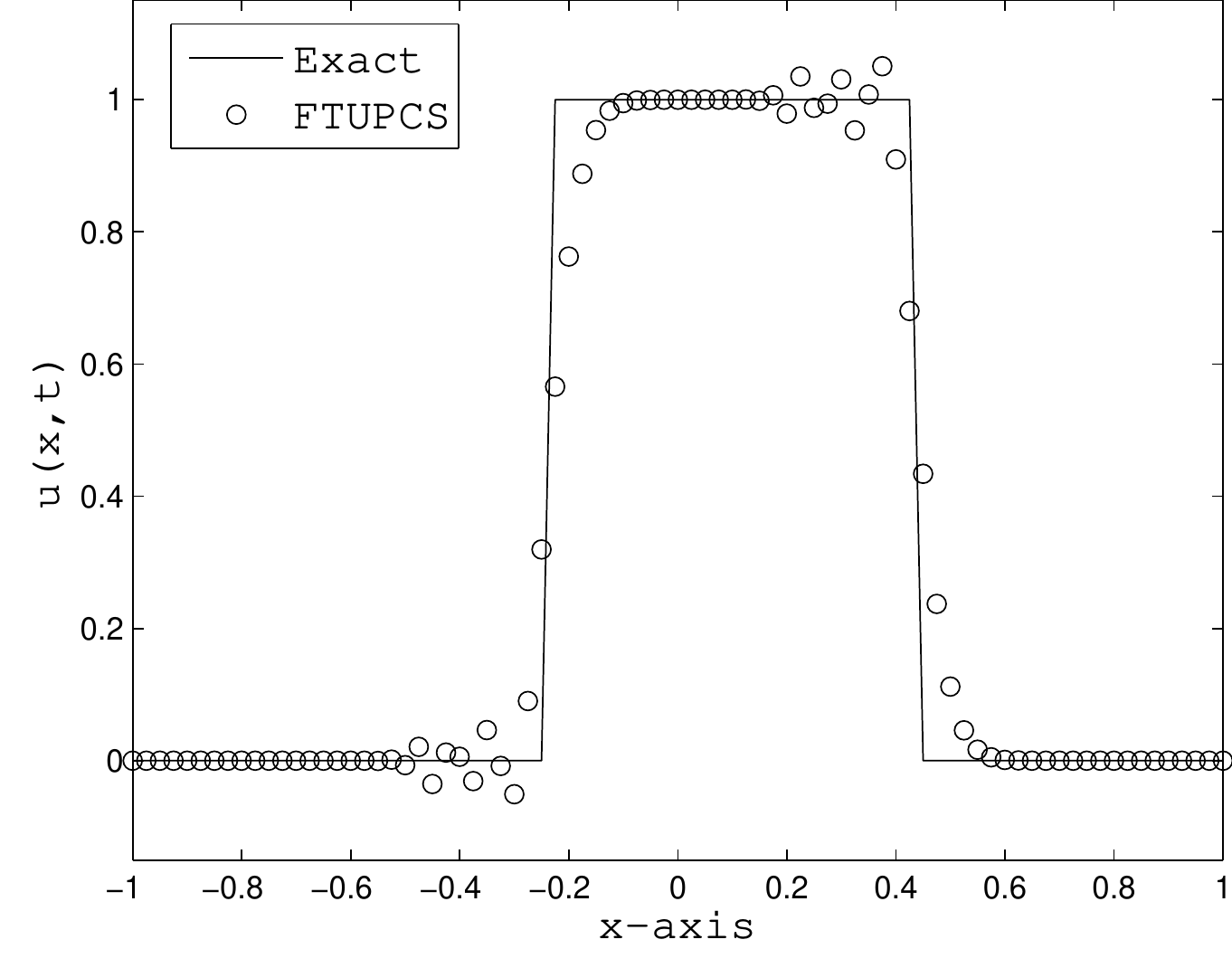}\\
(a) & (b) & (c)
\end{tabular}
\caption{\label{Fig2} Solution corresponding to (\ref{IC2}), 
  with $N=80$ grid point at $T=0.1$ and $C=0.1$, (a) FTCS does not introduce oscillations for top
  and bottom of left and right discontinuity respectively (b) No
  oscillations by FTCSUP, FTCS is applied only in its stability region
  and (c) Induced oscillations FTUPCS, FTCS is applied only in its
  unstability region.}
\end{center}
\end{figure}

\begin{figure}
\begin{center}
\begin{tabular}{cc}
\includegraphics[scale=0.4]{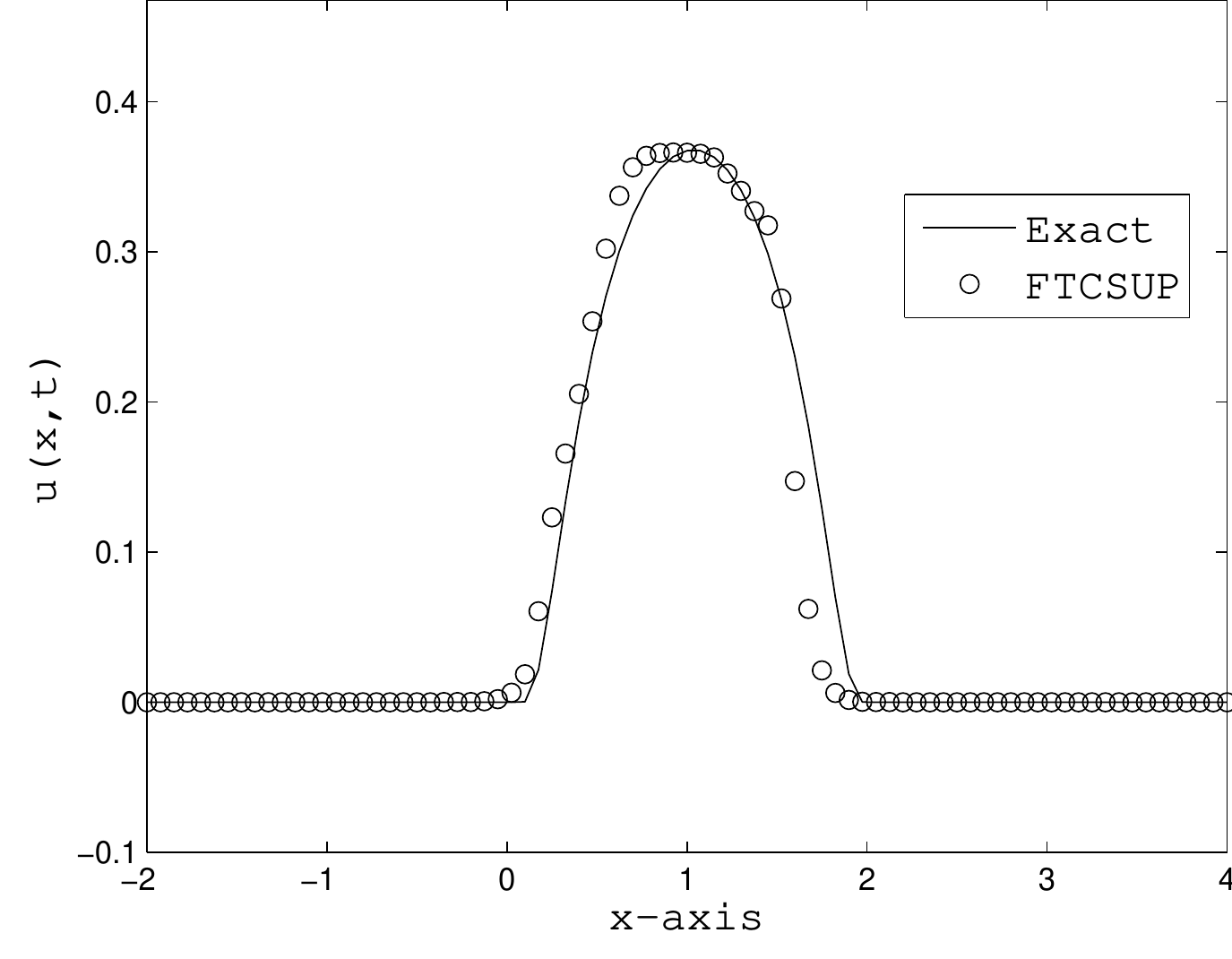}& \includegraphics[scale=0.4]{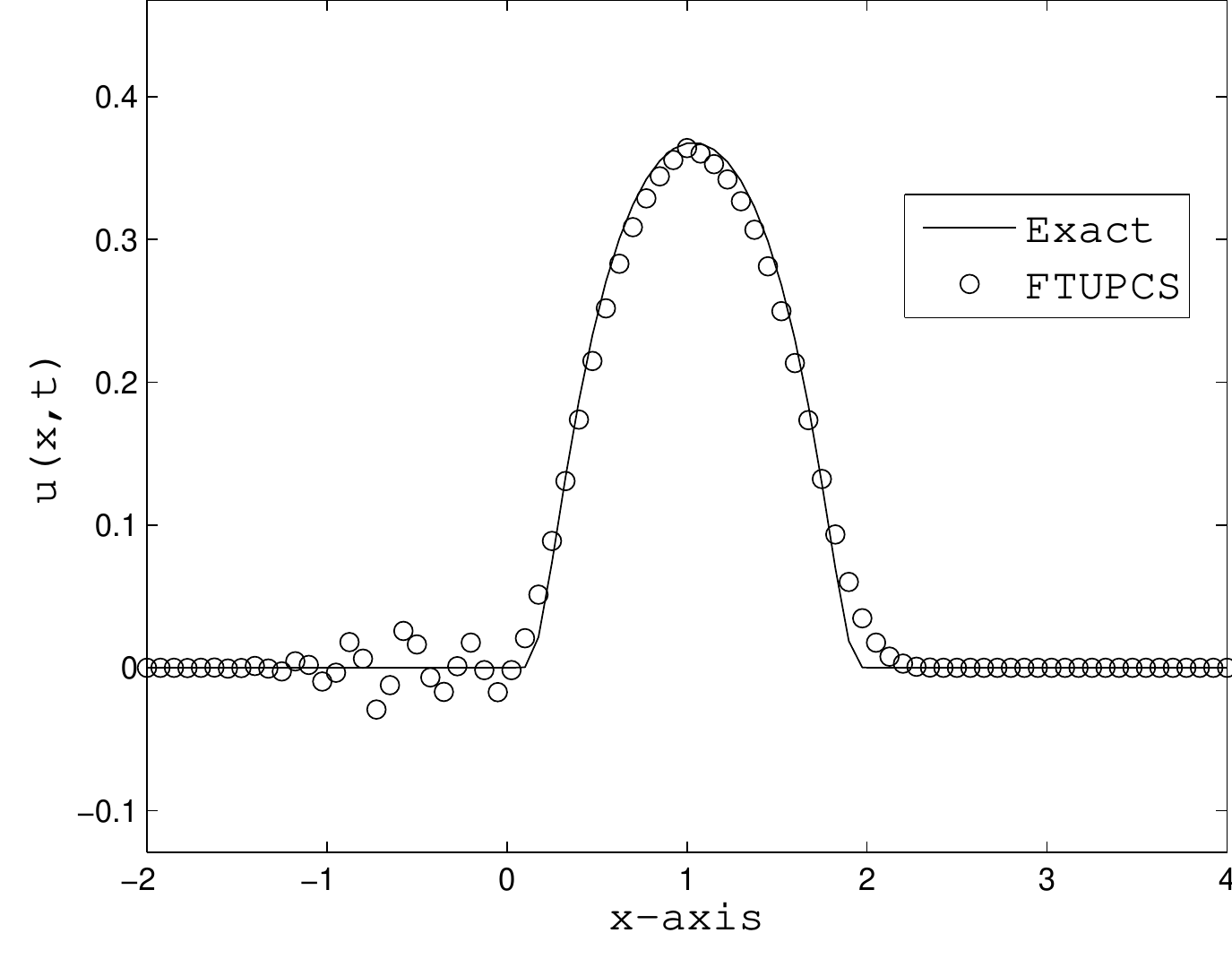}\\
(a) & (b)
\end{tabular}
\caption{\label{Fig3} Solution corresponding to (\ref{IC3}), at $T=1.0$ and $C=0.6$ (a) No
  oscillations by FTCSUP (b) Induced spurious oscillations FTUPCS, FTCS is applied only in its
  unstability region.} 
\end{center}
\end{figure}

\begin{figure}
\begin{center}
\begin{tabular}{ccc}
\includegraphics[scale=0.4]{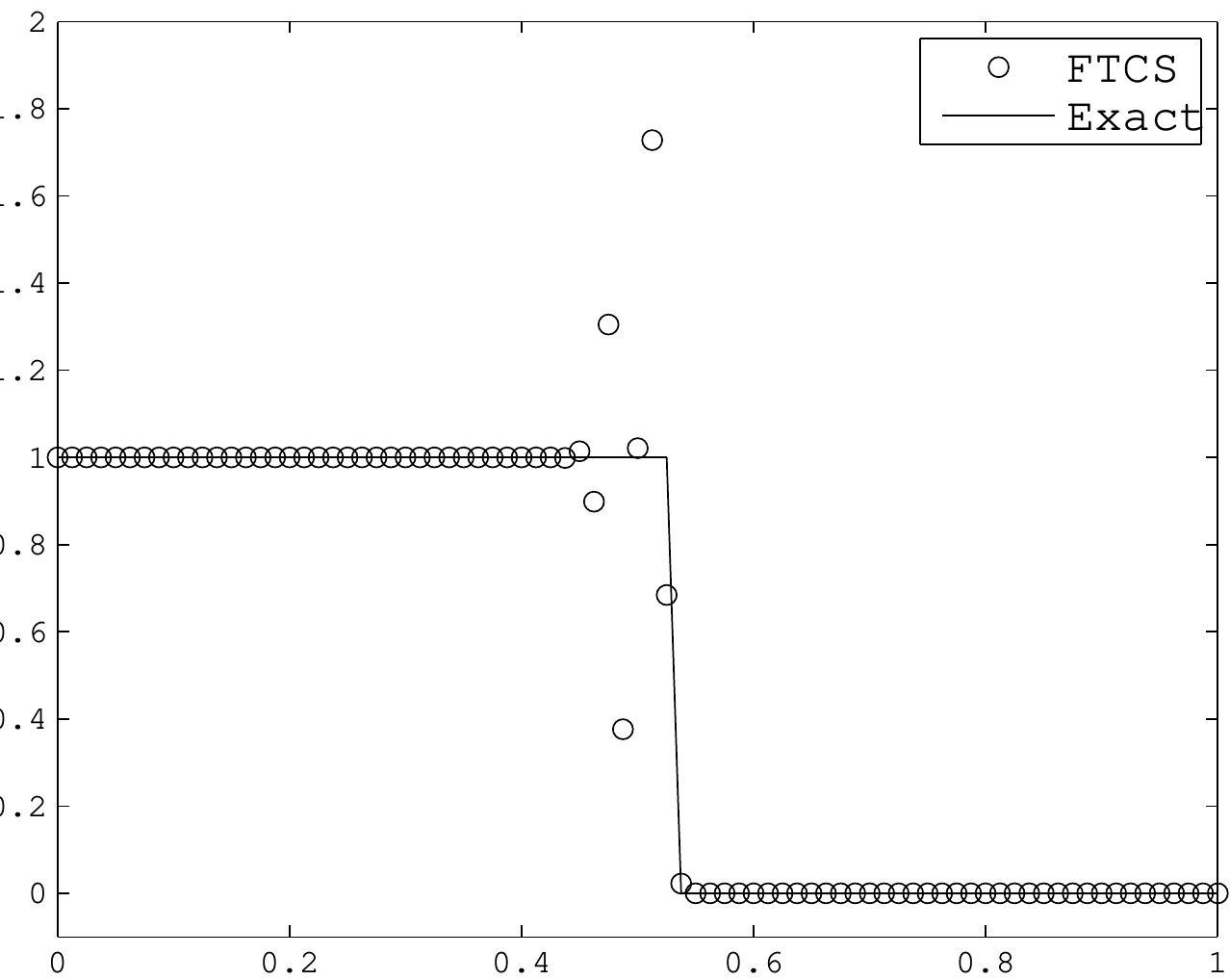}&\includegraphics[scale=0.4]{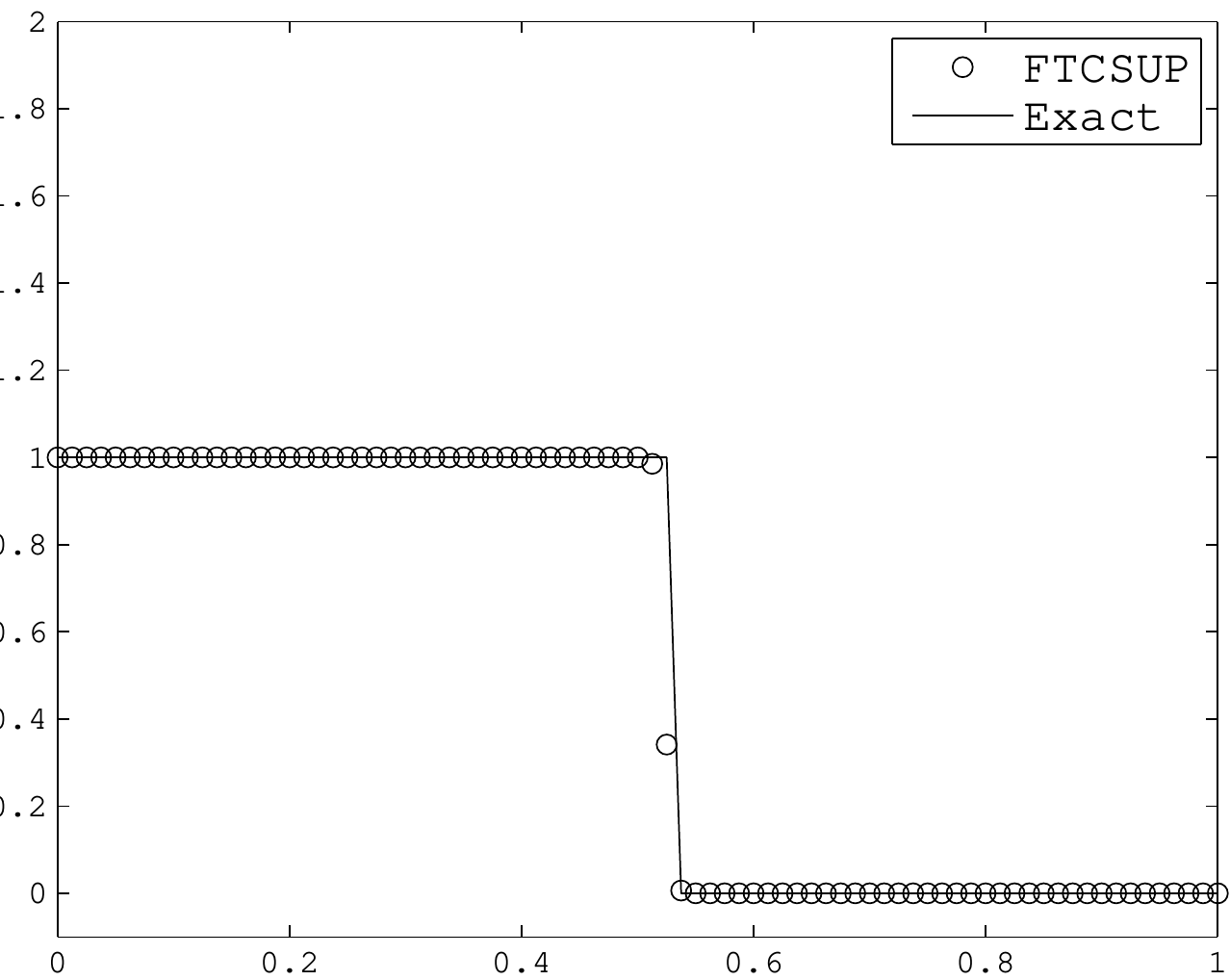} &\includegraphics[scale=0.4]{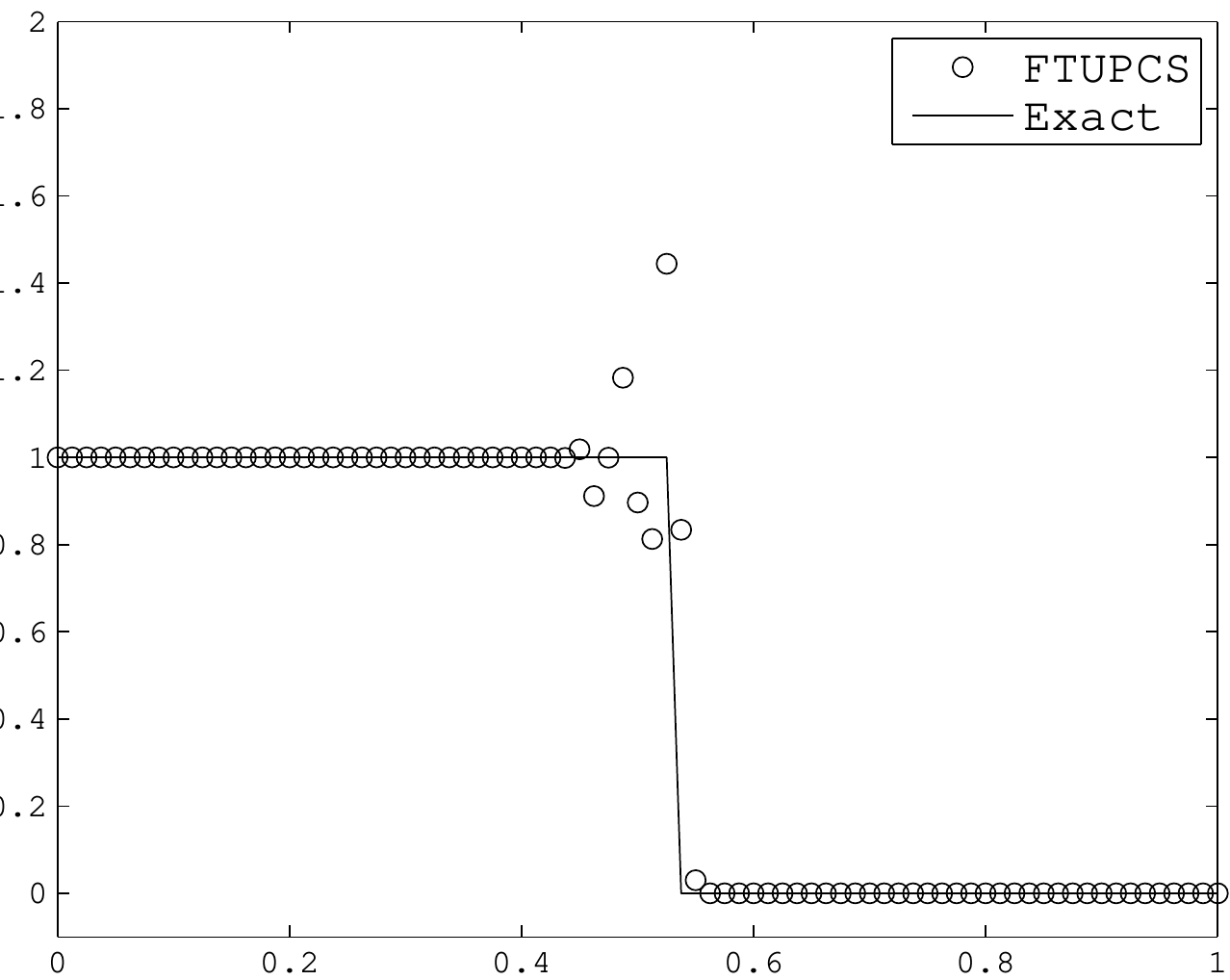}\\
(a) & (b) &(c)
\end{tabular}
\caption{\label{burgerFig1} Solution of Burgers equation corresponding to IC (\ref{test1IC}) using $CFL=0.9, N=80$, after 6 time steps at $T=0.05$}. 
\end{center}
\end{figure}
\begin{figure}
\begin{center}
\begin{tabular}{cc}
\includegraphics[scale=0.4]{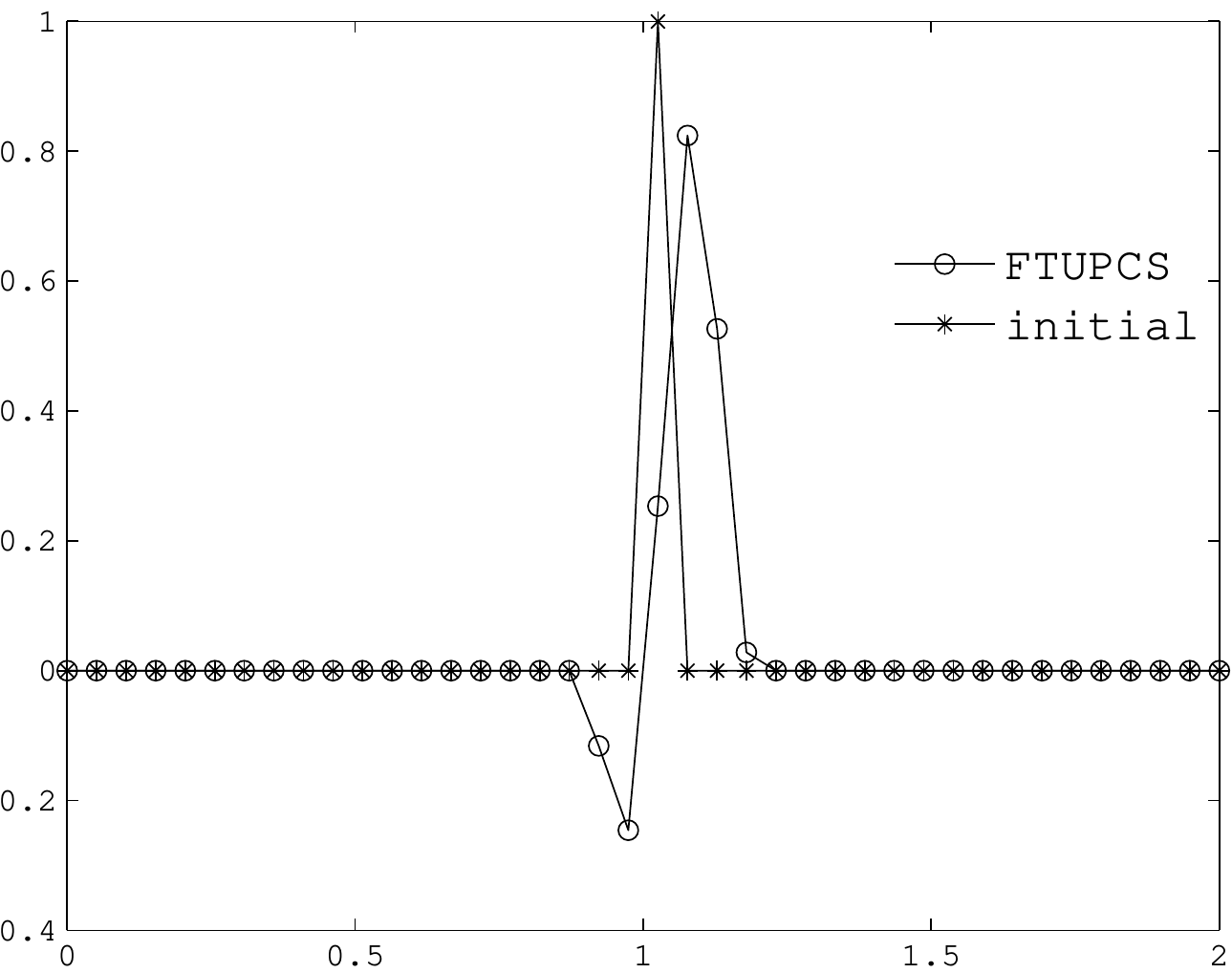}&\includegraphics[scale=0.4]{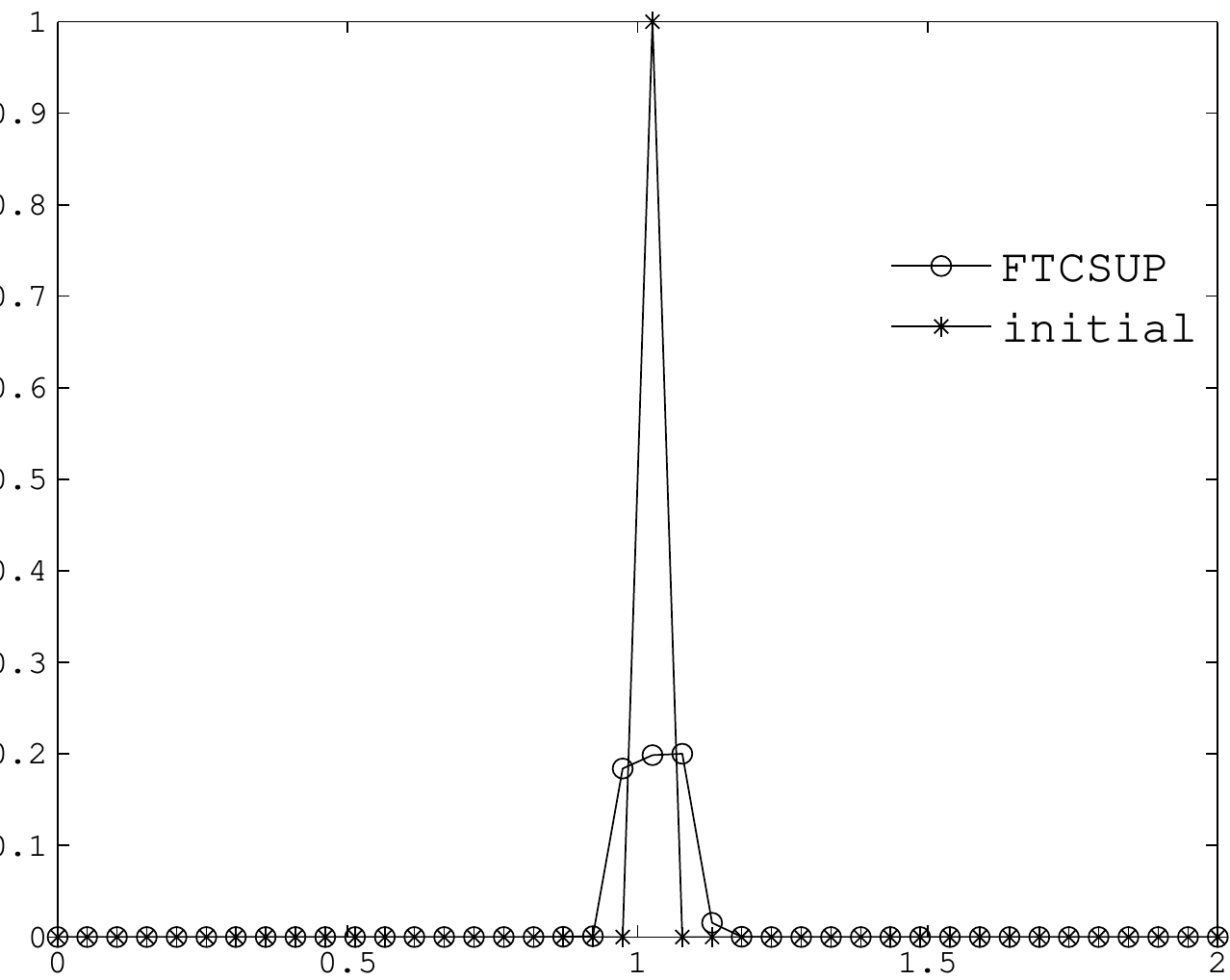}\\
(a) & (b) 
\end{tabular}
\caption{\label{burgerFig2} Solution of Burgers equation corresponding to IC (\ref{test1IC}) using $CFL=0.8, N=40$, after 3 time steps}. 
\end{center}
\end{figure}

\end{document}